\documentclass[12pt,reqno]{amsart}
\usepackage[widespace]{fourier}
\usepackage{hyperref}
\hypersetup{colorlinks=true,linkcolor=blue, citecolor=blue}
\usepackage{amsmath,amscd,amsfonts,amssymb}
\usepackage[all]{xypic}
\usepackage[utf8]{inputenc}
\usepackage{longtable}
\input xy
\xyoption{all}

\numberwithin{equation}{section}
\newtheorem{theorem}{Theorem}[section]
\newtheorem{definition}[theorem]{Definition}

\newtheorem{lemma}[theorem]{Lemma}
\newtheorem{proposition}[theorem]{Proposition}

\newtheorem{remark}[theorem]{Remark}

\begin{document}

\title{Intersection Poincar$\acute{\text{E}}$ Polynomial for Nagaraj-Seshadri moduli space}

\author[P. Barik]{Pabitra Barik}

\address{Department of Mathematics, Indian Institute of Technology-Madras, Chennai, India }

\email{pabitra.barik@gmail.com}

\author[A. Dey]{Arijit Dey}

\address{Department of Mathematics, Indian Institute of Technology-Madras, Chennai, India }

\email{arijitdey@gmail.com}

\author[Suhas, B N]{Suhas, B N}
\address{Department of Mathematics, Indian Institute of Technology-Madras, Chennai, India }
\email{chuchabn@gmail.com}

\maketitle

\abstract
We compute Betti numbers of both the components of the moduli space of rank $2$ semi-stable torsion-free sheaves with fixed determinant over a reducible nodal curve with two smooth components intersecting at a node. We also compute the intersection Betti numbers of the moduli space.
\endabstract
\section{Introduction}

Let $Y$ be a smooth irreducible projective curve over $\mathbb{C}$ degenerating to a reducible nodal curve $X$ with two smooth components meeting transversally at one point and $L$ be a line bundle of odd degree on $Y$ specializing to a line bundle $\xi$ on $X$. Then the moduli space $M_{L}$ of stable vector bundles of rank two and determinant $L$ specializes to the moduli space $M_{\xi}$ of stable torsion-free sheaves on $X$ of rank two and determinant $\xi$ \cite[Remark 7.1, 7.2]{N-S} . It is known that $M_{L}$ is a smooth irreducible projective variety.  The 'limiting' moduli space $M_{\xi}$, though, is neither smooth nor irreducible but has two smooth components $M_{12}$ and $M_{21}$, intersecting transversally along a smooth divisor $N$ \cite[Proposition 6.5]{S-B}.

The topological invariants of $M_{L}$ have been quite well studied and its Poincar$\acute{\text{e}}$ Polynomial has been computed by many authors using different techniques \cite{A-B}, \cite{U-R}, \cite{H-N}. It is natural to investigate similar topological invariants of $M_{\xi}$. Recently, it has been shown that the irreducible components $M_{12}$ and $M_{21}$ of $M_{\xi}$ are rational varieties (cf. \cite{B-D-S}). In particular, they are unirational and hence by a result of Serre \cite{Serre} they are simply connected. The first three Betti numbers of each component have been computed in \cite[Theorem 3.9]{S-B} by using topological techniques. In this paper, we find all the Betti numbers of both the components (see \S 4). We follow Harder--Narasimhan method \cite{H-N} which is based on Weil conjectures that tell us that for any smooth complex projective variety $V$ defined over a number ring, the number of $\mathbb F_{q^n}$ rational points of the corresponding variety $V$ over $\mathbb F_{q^n}$ determines the Betti numbers of $V$. In order to compute rational points of $M_{12}$ and $M_{21}$, we use their description in terms of triples \cite{N-S} and the stratification obtained in \cite{B-S}. The idea of proof is based on \cite{B-S}, where the authors compute the Betti numbers for Seshadri desingularization of the moduli space of rank $2$ semi-stable vector bundles with trivial determinant over a smooth complex projective curve.

Since the components $M_{12}$ and $M_{21}$ of $M_{\xi}$ intersect transversally along a smooth divisor $N$, the normalization $\widetilde{M}_{\xi}$ is isomorphic to $M_{12} \sqcup M_{21}$. Using this fact and a result of Goresky-Macpherson \cite[\S 4.2]{G-M} we compute the intersection Poincar$\acute{\text{e}}$ polynomial of $M_{\xi}$ (see Theorem \ref{IPP(M)}).

{\bf Acknowledgement.} It is a pleasure to thank V. Balaji for suggesting this problem and for many useful discussions we had while working on it. We also thank C. S. Seshadri and D.S. Nagaraj for some helpful discussions during this work.

\section{Preliminaries}

\subsection{A Brief Description of the Nagaraj-Seshadri Moduli Space}
In this section, we shall briefly recall some of the results proved in \cite{N-S} which will be useful in later sections. Let $X$  be a reducible projective nodal curve defined over an algebraically closed field with two smooth irreducible components $X_1$ and $X_2$ meeting at the nodal point $p$.  Any torsion-free sheaf $E$ on $X$ can be identified with a triple $(E_{1},
E_{2},\overrightarrow{T})$ or $( E_{1}^{\prime}, E_{2}^{\prime} , \overleftarrow{S} )$ where $E_{i}$'s are locally free sheaves over $X_{i}$'s and $\overrightarrow{T}$ and $
\overleftarrow{S}$ are linear maps from $E_{1}(p)$ to $E_{2}(p)$ and $E_{2}^{\prime}(p)$ to $E_{1}^{\prime}(p)$ respectively.  
\begin{definition}\label{definition_of_isomorphism_class_of_triples}
 (i) A morphism between two triples $(E_{1}, E_{2},\overrightarrow{T})$ and $(F_{1}, F_{2},\overrightarrow{U})$ is a pair $(h_1,h_2)$, where $h_i\,: E_i \rightarrow F_i$ is a map of $\mathcal{O}_{X_i}$-modules for each $i$, such that the following diagram commutes -
    $$\xymatrix{
E_{1}(p) \ar@{->}[r]^{T} \ar@{->}[d]^{h_1} & E_{2}(p) \ar@{->}[d]^{h_2}
\\F_{1}(p)\ar@{->}[r]^{U} & F_{2}(p).
}$$

(ii) If $h_i$'s are isomorphic as $\mathcal{O}_{X_i}$-modules, we say $(E_{1}, E_{2},\overrightarrow{T})$ and $(F_{1}, F_{2},\overrightarrow{U})$ are isomorphic as triples.
\end{definition}
We denote the isomorphism class of triple $(E_1,E_2,\overrightarrow{T})$ by $[E_1,E_2,\overrightarrow{T}]$. 
Similar definitions hold for triples in other direction. In fact in \cite[Lemma 2.3]{N-S}, an equivalence
is shown between the category of torsion-free sheaves and the category of triples. 

Let $E$ be a torsion-free sheaf on $X$ identified by the triple $(E_{1}, E_{2},
\overrightarrow{T})$ as well as $( E_{1}^{\prime}, E_{2}^{\prime} , \overleftarrow{S} )$. Then we have the following equality of Euler characteristics between them
(see\cite{N-S},Remark 2.11)-
\begin{equation}
 \chi(E) \,=\, \chi(E_{1}, E_{2},\overrightarrow{T})= \chi( E_{1}^{\prime}, E_{2}^{\prime} , \overleftarrow{S}),
\end{equation}

 where
 \begin{eqnarray}
  \chi(E_{1}, E_{2},\overrightarrow{T}):= \chi(E_{1}) + \chi(E_{2})- rk(E_{2}),
 \end{eqnarray}
  and
 \begin{equation}
  \chi( E_{1}^{\prime}, E_{2}^{\prime} , \overleftarrow{S}):= \chi( E_{1}^{\prime})+ \chi(E_{2}^{\prime}) - rk(E_{1}^{\prime}).
 \end{equation}

 Let ${\bf a}\,=\,(a_{1},a_{2})$ be a polarization on $X$ with $a_i \, > \,0$ rational numbers and $a_{1} + a_{2} \,=\, 1.$ For every non-zero triple $(E_{1},E_{2},\overrightarrow{T}),$ we define
  \begin{equation*}
   \mu((E_{1},E_{2},\overrightarrow{T}))\, = \,\frac{\chi(E_{1},E_{2},\overrightarrow{T})}{a_{1}rk(E_{1})+a_{2}rk(E_{2})}.
  \end{equation*}
  \begin{definition}
  Let $(E_{1},E_{2},\overrightarrow{T})$ be a triple. We say that the triple $(E_{1},E_{2},\overrightarrow{T})$ is stable (resp. semi-stable) if for every proper sub-triple $(G_{1},G_{2},\overrightarrow{U}),$ \\
  $\mu((G_{1},G_{2},\overrightarrow{U})) \,<\, \mu((E_{1},E_{2},\overrightarrow{T}))$ (resp. $\leq$).
  \end{definition}

(Similar definitions can be given for the triple $(E_{1}^{\prime},E_{2}^{\prime},\overleftarrow{S})$).

When $\chi$ is odd and $a_1\chi$ is not an integer the moduli space $M(2,{\bf a},\chi)$ of stable torsion-free sheaves on $X$ with Euler characteristic $\chi$ is a reduced, connected projective variety  with the two smooth irreducible components $M_{12}$ and $M_{21}$ intersecting transversally on a smooth projective variety $N$ \cite[Theorem 4.1]{N-S}. The irreducible components $M_{12}$ and $M_{21}$ have the following description in terms of triples:

 The first component $M_{12}$ is a smooth projective variety which is a fine moduli space consisting of isomorphism classes of stable triples $[E_{1}, E_{2},\overrightarrow{T}]$ such that $E_{i}$'s are rank 2 vector bundles over $X_{i}$'s, and $\overrightarrow{T}:\,E_{1}(p)\rightarrow E_{2}(p)$ is a nonzero linear map such that
 \begin{equation}\label{4}
  a_{1}\chi \,< \,\chi(E_{1}) \,< \,a_{1}\chi +1 ,~a_{2}\chi +1 \,< \,\chi(E_{2}) \,< \,a_{2}\chi +2 ,
 \end{equation}
The second component $M_{21}$ also has a similar description in terms of triples. It is a smooth projective variety which is a fine moduli space consisting of isomorphism classes of stable triples $[E_{1}^{\prime}, E_{2}^{\prime} , \overleftarrow{S}]$ such that $E_{i}^{\prime}$'s are rank 2 vector bundles over $X_{i}$'s, and $\overleftarrow{S}:\, E_{2}^{\prime}(p) \rightarrow E_{1}^{\prime}(p)$ is a nonzero linear map such that
\begin{equation}\label{5}
  a_{1}\chi +1 < \chi(E_{1}^{\prime}) < a_{1}\chi +2 ,~a_{2}\chi < \chi(E_{2}^{\prime}) < a_{2}\chi +1.
 \end{equation}

 Form now on, we fix an odd integer $\chi$ and a polarization ${\bf a}\,=\,(a_1,a_2)$ such that $a_1\chi$ is not an integer.
 \begin{lemma} \label{divisor}
 Let $E_1$ be a rank $2$ semi-stable vector bundle on $X_1$ and $E_2$ be a rank $2$ semi-stable  vector bundle on $X_2$ such that at least one of them is stable and
$\chi(E_1)$, $\chi(E_2)$ satisfy inequalities \eqref{4}. Then for any nonzero linear map $\overrightarrow{T}\,:E_1(p) \rightarrow E_2(p)$, the triple $(E_1,E_2,\overrightarrow{T})$ is stable.
 \begin{proof}
One can easily prove this lemma by a case-by-case analysis.
\end{proof}
\begin{remark}\label{remark_after_divisor_lemma}
 The converse of the above Lemma is also true, i.e., if $(E_1,E_2,\overrightarrow{T})$ is a rank two, semi-stable triple, such that $\chi(E_1)$ and $\chi(E_2)$ satisfy inequalities \eqref{4}, then $E_1$ and $E_2$ are semi-stable vector bundles on $X_1$ and $X_2$ respectively \cite[Theorem 5.1]{N-S}. Similar results hold for triples in the other direction also.

\end{remark}
\end{lemma}
The intersection subvariety $N$ is given by
\begin{eqnarray*}
   \{[E_{1},E_{2},\overrightarrow{T}] \in M_{12}~|~\text{rk}~(\overrightarrow{T}) \,=\, 1\},
\end{eqnarray*}
which can be identified with
\begin{eqnarray*}
 \{[E_{1}^{\prime},E_{2}^{\prime},\overleftarrow{S}] \in M_{21}~|~\text{rk}~(\overleftarrow{S}) \,=\, 1\}.
\end{eqnarray*}
Further $N$ can be described as a product $N_1 \times N_2$, where $N_1$ and $N_2$ are rank $2$ parabolic moduli spaces over $X_1$ and $X_2$ with Euler characteristics $\chi_1$ and $\chi_2$ respectively satisfying the inequalities \eqref{4}, having full flag and parabolic weights $(a_1/2,a_2/2)$. In characteristic $0$, this was proved in \cite[Theorem 6.1]{N-S}. Here we will give a characteristic-free proof.
\begin{proposition}
$N \,\simeq\, N_1 \times N_2$
\end{proposition}
\begin{proof}
Let $\mathbf F_i$ be the universal quotient sheaf on $X_i \times R_i$ parametrizing vector bundles with Euler characteristic $\chi_{i}$, where $R_i$ is the quot scheme parametrizing locally free quotients satisfying certain cohomological conditions (see \cite[Theorem 5.3]{N-S} for details). Let $F_i\,=\, p_{i}^{*} (\mathbf F_i\mid_{\{p\}\times R_i})$ on $R_1 \times R_2$, where $p_i$'s are projections from $R_1 \times R_2$ to $R_i$'s. Consider the
projective bundle $\mathbb P(F_1^*)$. The fibre over any closed point $(q_1,q_2)$ parametrizes all surjective $1$-dimensional quotients $\{F_1(q_1) \rightarrow W(q_1)\}$.
Similarly, for the projective bundle $\mathbb P(F_2)$, the fiber over any closed point $(q_1,q_2)$ parametrizes all $1$ dimensional subspaces $\{V(q_2) \subset F_2(q_2)\}$.
Let $\mathbb P(F_1^*)^{ss}$ (resp. $\mathbb P(F_2)^{ss}$) denote the set of closed points of $\mathbb P(F_1^*)$ (resp. $\mathbb P(F_2)$) such that corresponding parabolic bundle on $X_1$ (resp. on $X_2$) is parabolic semi-stable (= parabolic stable)\footnote{Since we are in rank $2$ and there is only one parabolic point the parabolic stability coincides with parabolic semi-stability} with respect to parabolic weights $(\frac{a_1}{2},\frac{a_2}{2})$. Let $((F_1(q_1) \rightarrow W(q_1)\, ,\,V(q_2) \subset F_2(q_2))$ be a point in $\mathbb P(F_1^*)^{ss} \times \mathbb P(F_2)^{ss}$ and $\lambda:\,W(q_1) \simeq V(q_2)$ be an isomorphism. By composing $\lambda$ with the quotient $F_1(q_1) \rightarrow W(q_1)$, we get a rank $1$ linear map  $T_\lambda$ from $F_1(q_1) \rightarrow F_2(q_2)$. This gives rise to a triple $(\mathbf F_{1,q_1},\mathbf F_{2,q_2}, \overrightarrow{T_{\lambda}})$. From the
choice of parabolic weights one can easily see that $\mathbf F_{1,q_1}$ and $\mathbf F_{2,q_2}$ are semi-stable on $X_1$ and $X_2$ respectively and at least one of them is stable. So by Lemma \ref{divisor}, the triple $(\mathbf F_{1,q_1},\mathbf F_{2,q_2}, \overrightarrow{T_{\lambda}})$ is semi-stable. Since there is nothing canonical in $\lambda$, we naturally get an element in $\mathbb P(\mathcal Hom(F_1,F_2))^{ss}$. Hence we have a commutative diagram
$$
\xymatrix{
\mathbb P(F_1^*)^{ss} \times_{R_1 \times R_2} \mathbb P(F_2)^{ss} \ar[r]^{C} \ar[dr] & \mathbb P(\mathcal Hom(F_1,F_2))^{ss}\ar[d] \\
&R_1 ^{ss} \times R_2^{ss}.
}
$$
There is a natural action of the group $G\,=\,PGl(k_1) \times PGl(k_2)$ (for a suitable choice of $k_1$ and $k_2$) on both the domain and codomain of $C$ and since $C$ is $G$-equivariant, the morphism $C$ descends to respective quotients. So we have the following commutative diagram
$$
\xymatrix{
N_1 \times N_2 \ar[r]^{\overline{C}} \ar[dr] & M_{12} \ar[d] \\
& M_1 \times M_2,
}
$$
where $M_1$ (resp. $M_2$) is the moduli space of rank $2$ stable (resp. semi-stable) vector bundles with Euler characteristic $\chi_1$ (resp. $\chi_2$) over $X_1$ (resp. $X_2$). It is clear that the image of $\overline{C}$ is inside $N$. Further  there is a natural map $D:\,N \rightarrow N_1 \times N_2$ which sends any semi-stable triple
$(E_1,E_2,\overrightarrow{T})$ to $((E_1,  E_1(p) \rightarrow E_1(p)/ \text{Ker}(\overrightarrow{T})), (E_2, 0\subset \text{Image}(\overrightarrow{T}) \subset E_2(p)))$ (see also \cite[Theorem 6.1]{N-S}). One can
easily check that $D \circ \overline{C}\,=\,\overline{C} \circ D\,=\,\text{id}$.
\end{proof}

$\mathbf{The~notion~of~fixed~determinant~moduli~space :}$

Let $E$ be an element of $M(2,{\bf a},\chi)$ and $(E_{1},E_{2},\overrightarrow{T})$, $(E_{1}^{\prime},E_{2}^{\prime},\overleftarrow{S})$ be the triples corresponding to $E$. Let $\chi_{i}:= \chi(E_{i})$ and $\chi_{i}^{\prime}:= \chi(E_{i}^{\prime}),$ for $i \,=\, 1,2.$ If $\chi_{i}$'s satisfy the inequalities in \eqref{4}, then $E$ belongs to $M_{12}$ and if $\chi_{i}^{\prime}$'s satisfy the inequalities in \eqref{5}, then $E$ belongs to $M_{21}.$ Let $J^{d_{i}}(X_{i})$ be the Jacobian of line bundles of \text{\text{deg}}ree $d_{i}$ on $X_{i},$ where $d_{i} \,=\, \chi_{i}-2(1-g_{i}),$ for $i \,=\, 1,2.$ By \cite[Proposition 7.1]{N-S}, there is a well-defined surjective morphism
\begin{eqnarray*}
 \text{det}\, : \,M(2,{\bf a},\chi \neq 0) & \longrightarrow & J^{d_{1}}(X_{1}) \times J^{d_{2}}(X_{2}),
\end{eqnarray*}
given by
\begin{eqnarray*}
E      & \mapsto &  (\Lambda^{2}(E_{1}),\Lambda^{2}(E_{2})) \,\,\hspace{.3cm} \text{if}  \,\,\hspace{.3cm} E \in M_{12}, ~\,\,\,\text{and} \\
E      & \mapsto &  \Psi((\Lambda^{2}(E_{1}^{\prime})),\Lambda^{2}(E_{2}^{\prime})) \,\,\hspace{.3cm} \text{if} \,\,\hspace{.3cm} E \in M_{21},
\end{eqnarray*}
where
\begin{equation*}
 \Psi:J^{d_{1}+1}(X_{1}) \times J^{d_{2}-1}(X_{2})  \rightarrow  J^{d_{1}}(X_{1}) \times J^{d_{2}}(X_{2})
\end{equation*}
is an isomorphism defined by
\begin{equation*}
 (L_{1},L_{2})  \mapsto  (L_{1} \otimes \mathcal{O}_{X_{1}}(-p),L_{2} \otimes \mathcal{O}_{X_{2}}(p)).
\end{equation*}
Now let us fix $L_{1} \in J^{d_{1}}(X_{1})$ and $L_{2} \in J^{d_{2}}(X_{2})$, we write ${\bf \xi}\,=\,(L_1,L_2)$. Then the fixed determinant moduli space $M(2,\mathbf{a},\chi,\xi)$ is by definition $\text{det}\,^{-1}({\bf \xi}).$ By \cite[Proposition 7.2]{N-S}, it is reduced.
Let $\text{det}\,_{12}$ (resp. $\text{det}\,_{21}$) be the morphism $\text{det}\,|_{M_{12}}$ (resp. $\text{det}\,|_{M_{21}}$). For notational convenience we write $M_{12}({\bf \xi})$ (resp. $M_{21}({\bf \xi})$) for $\text{det}\,_{12}^{-1}({\bf \xi})$ (resp. $\text{det}\,_{21}^{-1}({\bf \xi})$).
Then we have
\begin{eqnarray*}
 M_{12}({\bf \xi})  & = & \{[E_{1},E_{2},\overrightarrow{T}] \in M_{12} ~|~ \Lambda^{2}(E_{1}) \,\simeq\, L_{1},~\Lambda^{2}(E_{2}) \,\simeq\, L_{2} \},
\end{eqnarray*}
and
\begin{eqnarray*}
M_{21}({\bf \xi})   & = & \{[E_{1}^{\prime},E_{2}^{\prime},\overleftarrow{S}] \in M_{21} ~|~ \Lambda^{2}(E_{1}^{\prime}) \,\simeq\, L_{1} \otimes \mathcal{O}_{X_{1}}(p),\Lambda^{2}(E_{2}^{\prime}) \,\simeq\, L_{2} \otimes \mathcal{O}_{X_{2}}(-p)) \}.
\end{eqnarray*}
By \cite[Proposition 6.5]{S-B}, the fixed determinant moduli space is a connected, projective scheme with exactly two smooth irreducible components $M_{12}({\bf \xi})$ and $M_{21}({\bf \xi})$, meeting transversally along the smooth divisor $N({\bf \xi}) \,=\, M_{12}({\bf \xi}) \cap N$ (which is identified with $M_{21}({\bf \xi}) \cap N$). Since $\chi$ is assumed to be an odd integer, and $\chi \,=\, \chi_{1} + \chi_{2}-2,$ we can conclude that either $\chi_{1}$ is odd or $\chi_{2}$ is odd and not both (Same argument applies to $\chi_{1}^{\prime}$ and $\chi_{2}^{\prime}$ also).
\subsection{The Intersection Homology}

Suppose $Y$ is a compact oriented manifold such that $\text{dim}~_{\mathbb{R}}Y \,=\, n.$ Then Poincar$\acute{\text{e}}$ duality states  that there exists a perfect pairing
\begin{equation}\label{Poincare_duality}
 H_i(Y) \otimes H_{n-i}(Y) \rightarrow \mathbb{C},
\end{equation}
for each $0 \leq i \leq n,$ and in particular
\begin{equation*}
 \text{dim}_{\mathbb{C}}H_i(Y) \,= \, \text{dim}_{\mathbb{C}}H_{n-i}(Y).
\end{equation*}
It is known that Poincar$\acute{\text{e}}$ duality does not hold for the cohomology of singular spaces (cf. \cite[pp. 3-4]{K-W} for examples). To remedy this, a generalization to singular spaces of the Poincar$\acute{\text{e}}$-Lefschetz theory of intersections of homology cycles on manifolds, was introduced by Goresky and Macpherson in \cite{G-M} and \cite{G_M-II}. Their main objective was to study the intersection of cycles on an $n$ dimensional oriented pseudomanifold.

\begin{definition}[\cite{G_M-II} \S 1.1]
 A topological pseudomanifold of dimension $n$ is an $n$-dimensional stratified paracompact Hausdorff topological space $Y$ which admits a
stratification
\begin{equation}\label{stratification}
 Y \,=\, Y_n \supset Y_{n-1} \supset Y_{n-2} \supset \cdots \supset Y_1 \supset Y_0
\end{equation}
such that $Y_{n-1} \,=\, Y_{n-2}$ and $Y-Y_{n-1}$ is dense in $Y.$
\end{definition}
For a topological pseudomanifold of dimension $n,$ they introduced a new class of groups $IH_{*}^{\overline{p}}(Y)$ called the intersection homology groups. These groups depend on the choice of a perversity $\overline{p}$ which is a function from the set $\{2,\dots,n\}$ to $\mathbb{N}$ such that $\overline{p}(2) \,=\, 0$ and $\overline{p}(i+1) \,=\, \overline{p}(i)$ or $\overline{p}(i)+1$ (cf. \cite[Definition 4.2.3]{K-W}).

%
%
%

Further, if $\overline{p}, \overline{q}, \text{and}~ \overline{r}$ are perversities such that $\overline{p}+\overline{q} \leq \overline{r},$ then there is a unique intersection pairing
 \begin{equation}\label{intersection_pairing}
  \cap: IH_{i}^{\overline{p}}(Y) \times IH_{j}^{\overline{q}}(Y) \rightarrow IH_{i+j-n}^{\overline{r}}(Y)
 \end{equation}
 such that $[C \cap D] \,=\, [C] \cap [D]$, for every dimensionally transverse pair of cycles $C \in IH_{i}^{\overline{p}}(Y)$ and $D \in IH_{j}^{\overline{q}}(Y)$ (cf. \cite[\S 2.1 Definition and \S 2.3 Theorem 1]{G-M}).

 Moreover, if $i$ and $j$ are complementary dimensions ($i+j \,=\, n$), $\overline{t}$ is the top perversity (which is defined to be the function $i \mapsto (i-2)$ for $2 \leq i \leq n$) and if $\overline{p}$ and $\overline{q}$ are complementary perversities ($\overline{p}+\overline{q} \,=\, \overline{t}$) then the pairing
 \begin{equation*}
  IH_{i}^{\overline{p}}(Y) \times IH_{j}^{\overline{q}}(Y) \xrightarrow{\cap}  IH_{0}^{\overline{t}}(Y) \xrightarrow{\epsilon} \mathbb{Z}
 \end{equation*}
 is a perfect pairing when tensored with rationals $\mathbb{Q},$ where
 \begin{equation*}
  \epsilon : IH_{0}^{\overline{t}}(Y) \rightarrow \mathbb{Z}
 \end{equation*}
 is the "augmentation" which counts points with multiplicities (cf. \cite[\S 3.3 Theorem]{G-M}).

It is also known that these intersection homology groups $IH_{*}^{\overline{p}}(Y)$ of a topological pseudomanifold $Y$ are topological invariants . This fact is proved using the sheaf theoretic construction of the intersection complex $\mathbf{I\dot{C}}$ (cf.\cite[\S 3 and \S 4]{G_M-II}).

Another important property of intersection homology groups is their behavior under normalization. The following theorems will be very useful in \S 5 where we compute the intersection homology groups of the fixed determinant moduli space $M(2,\mathbf{a},\chi,\xi)$.

\begin{theorem}[\cite{G-M} \S 4.2]\label{(G_M)_Theorem_4.2}
 If $Y$ is a pseudo-manifold with normalization $\pi: \tilde{Y}\rightarrow Y$, then the map $\pi$ induces isomorphisms $IH_{i}^{\overline{p}}(\tilde{Y})\,\simeq\, IH_{i}^{\overline{p}}(Y)$ for any perversity $\overline{p}.$
\end{theorem}

\begin{theorem}[\cite{G-M} \S 4.3]\label{(G-M)_Theorem_4.3}
 Let $Y$ be a topologically pseudomanifold of dimension $n.$ If $Y$ is topologically normal then there are canonical isomorphisms $IH_{i}^{\overline{t}}(Y)\,\simeq\, H_{i}(Y)$ and $IH_{i}^{\overline{0}}(Y) \,\simeq\, H^{n-i}(Y),$
 where $\overline{t}$ is the top perversity, $\overline{0}$ is the zero perversity, $H_{i}(Y)$ and $H^{n-i}(Y)$ stand for the usual singular homology and cohomology groups of $Y.$
\end{theorem}

An important example of a pseudomanifold (which we shall be concerned with in this article) is that of a complex quasi-projective variety of pure dimension $n.$ It is a theorem of Whitney (\cite[Theorem 19.2]{Wh}) that any quasi-projective variety $Y$ of pure dimension $n$ has a Whitney stratification (cf. \cite[Definition 4.10.2]{K-W}). It is this stratification of $Y$ which makes it into a topological pseudomanifold (\cite [IV \S 2]{Bo}).

\begin{definition}\label{definition_of_IPP}
Let $Y$ be a complex quasi-projective variety of pure dimension $n.$ We define the intersection Poincar$\acute{e}$ polynomial of $Y$ with respect to the perversity $\overline{p}$ to be the polynomial $\sum_{j=0}^{2n}\beta_{j}^{\overline{p}}t^{j},$ where $\beta_{j}^{\overline{p}}=\text{dim}_{\mathbb{R}}(IH_{j}^{\overline{p}}(Y)).$
\end{definition}
%
%
\begin{remark}\label{normalization_related_remark}
  If $Y$ is a complex quasi-projective variety  then by \cite[\S 4.1  pp. 151]{G-M} , it is known that the algebraic normalization $\widetilde{Y_{alg}}$ of $Y$ is homeomorphic to the topological normalization $\tilde{Y}$ of $Y.$ So by the Theorem \ref{(G_M)_Theorem_4.2}, in order to compute the intersection homology groups for such a $Y,$ it is enough to compute it for $\widetilde{Y_{alg}}.$
\end{remark}

\subsection{Notation} \label{notation}
From now on, we will drop some extra baggages and proceed with the following notations -

$\bullet$ $X$ denotes the projective reducible nodal curve with two smooth components $X_{1}$ and $X_{2}$ meeting at the nodal point $p.$

$\bullet$ $\xi=(L_1,\mathcal{O}_{X_2}),$ where $L_1$ is an invertible sheaf on $X_1$ of \text{\text{deg}}ree $1.$

$\bullet$  By $M$ we mean the moduli space of rank 2, Euler-characteristic $3-2g$, stable torsion free sheaves over $X$ with the fixed determinant $\xi.$

$\bullet$ The moduli space $M_{12}=\{[E_{1},E_{2},\overrightarrow{T}] :~[E_{1},E_{2},\overrightarrow{T}]~\text{is an isomorphism class of stable triples}\\ \text{with the respective Euler characteristics satisfying inequalities \eqref{4}}, ~\text{rk}~(E_{i})=2, \\ \Lambda^{2}(E_{1})\,\simeq\,L_{1}, \Lambda^{2}(E_{2})\,\simeq\,\mathcal{O}_{X_{2}},~\text{deg}~(L_{1})=1 \}.$

$\bullet$ The moduli space $M_{21}=\{[E_{1}^{\prime},E_{2}^{\prime},\overleftarrow{S}] :~[E_{1}^{\prime},E_{2}^{\prime},\overleftarrow{S}]~\text{ is an isomorphism class of stable triples}\\ \text{with the respective Euler characteristics satisfying inequalities \eqref{5}}, ~\text{rk}~(E_{i}^{\prime})=2, \\
\Lambda^{2}(E_{1}^{\prime})\,\simeq\,L_{1} \otimes \mathcal{O}_{X_1}(p), \Lambda^{2}(E_{2}^{\prime})\,\simeq\,\mathcal{O}_{X_{2}}(-p)\}.$

$\bullet$ Let $N = M_{12} \cap M_{21}$ and $M_{12}^{0} = M_{12} \backslash N.$

$\bullet$ The moduli space $M_{1}=\{ [E_{1}] :~ [E_{1}]~\text{is the isomorphism class of stable bundles on}~X_{1},\\
          ~\text{rk}~(E_{1})=2,~\Lambda^{2}(E_{1})\,\simeq\,L_{1}\}.$

$\bullet$ The moduli space $M_{2}=\{[E_{2}]:~[E_{2}]~\text{is the gr-equivalence class of semi-stable vector bundles} \\ \text{on}~X_{2}, ~\text{rk}~(E_{2})=2,~\Lambda^{2}(E_{2})\,\simeq\, \mathcal{O}_{X_{2}}\}.$

$\bullet$ $N_{1} = \{E_{1_{*}}:=(E_{1}, 0 \subset F^{2}E_{1}(p) \subset E_{1}(p)) :~E_{1_{*}}~\text{is parabolic semi-stable bundle with respect to} \\ \text{the parabolic weight}~(\frac{a_{1}}{2},\frac{a_{2}}{2})~ \text{on}~X_{1},~\text{where} ~0<a_{1}<a_{2}<1,~\text{rk}~(E_{1})=2,~\Lambda^{2}(E_{1})\,\simeq\,L_{1}\}.$

$\bullet$ $N_{2} = \{E_{2_{*}}:=(E_{2}, 0 \subset F^{2}E_{2}(p) \subset E_{2}(p)) :~E_{2_{*}}~\text{is parabolic semi-stable bundle with respect to}\\ \text{the parabolic weight}~(\frac{a_{1}}{2},\frac{a_{2}}{2})~\text{on}~X_{2},~\text{where}~0<a_{1}<a_{2}<1,~\text{rk}~(E_{2})=2,~\Lambda^{2}(E_{2})\,\simeq\, \mathcal{O}_{X_{2}}\}.$

\section{Counting $\mathbb{F}_{q}$ rational points}
Throughout this section, we assume our base field to be an algebraically closed field $\overline{\mathbb F_{l}}$, where $l$ is a prime number. Let $M$ be the moduli space of rank $2$ stable torsion free sheaves on $X$ with Euler characteristic $\chi$ and fixed determinant $\xi$ as in Notation \ref{notation} such that $a_1\chi$ is not an integer. By \S 2, $M$ is a reduced, connected projective variety  with two smooth irreducible components $M_{12}$ and $M_{21}$ intersecting transversally on a smooth projective variety $N$. Hence without loss of generality, we can assume both $X$ and $M$ are defined over a finite field extension $\mathbb F_q$ of $\mathbb F_l$.

In this section, we are interested in computing $\mathbb F_q$-rational points of $M_{12}$ and $M_{21}$ respectively. We will explicitly compute it for the component $M_{12}$, the computation for other component follows similarly. Now we consider the following morphism-
\begin{equation} \label{fi}
  \Phi : M_{12} \rightarrow M_{1} \times M_{2}
\end{equation}
which sends the equivalence class of triple $[E_{1}, E_{2},\overrightarrow{T}]$ in $M_{12}$ to equivalence class of pairs $([E_{1}], [E_{2}])$ in $M_1 \times M_2$. By Lemma \ref{divisor} and Remark \ref{remark_after_divisor_lemma}, $\Phi$ is well defined and surjective. A priori, $\Phi$ is defined over $\overline{\mathbb {F}}_l$ but without loss of generality, we can assume that it is defined over $\mathbb F_{q}$.
 We have the following stratification on $M_{1} \times M_{2}$-
 \begin{equation}\label{stratification1}
 M_{1} \times M_{2} \, = \, (M_{1} \times M_{2}^{s}) \sqcup (M_{1} \times (K-K_{0})) \sqcup (M_{1} \times K_{0}),
\end{equation}
 where
 $M_{2}^{s}$ is the stable locus of vector bundles over $X_{2}$ with fixed determinant $\mathcal{O}_{X_{2}}$, $K$ is the Kummer variety associated to the Jacobian $J_{2}$ of $X_{2}$ which can be identified with the isomorphism classes of vector bundles of the form $(L \oplus L^{-1})$
 where L is a line bundle of degree 0 over $X_{2},$
 $$ K-K_{0} \, = \, \{ [L \oplus L^{-1}] ~:~ L \in J_{2} ,~  L^{2}~ \text{not isomorphic to}~ \mathcal{O}_{X_{2}} \},$$ and
 $$ K_{0} \, = \, \{[L \oplus L] ~:~ L \in J_{2} ,~  L^{2} \,\simeq\, \mathcal{O}_{X_{2}} \}.$$
 Since $K$ is a projective variety, without loss of generality we can assume $K$ is defined over $\mathbb{F}_{q}$.
It is also known that $K_{0}$ is a finite set having $2^{2g_{2}}$ points.
 Now since
 $$M_{12} \, = \, M_{12}^{0} \sqcup N ,$$ where $M_{12}^{0}$ and $N$ are as in \S \ref{notation},
we have $$N_{q}(M_{12}) \, = \, N_{q}(M_{12}^{0}) + N_{q}(N),$$
 where, by $N_{q}(V)$ we mean the number of $\mathbb{F}_{q}$-rational points of a quasi-projective variety $V$ defined over $\overline{\mathbb{F}}_{q}$.

 Denoting the restriction of the map $\Phi$ to the open set $M_{12}^{0}$ by $\phi$ and using the stratification \eqref{stratification1}, we obtain
 \begin{eqnarray}\label{Equation 3.2}
   N_{q}(M_{12}) & = & N_{q}(\phi^{-1}(M_{1}\times M_{2}^{s})) + N_{q}(\phi^{-1}(M_{1} \times (K-K_{0}))) \nonumber \\
   & & + N_{q}(\phi^{-1}(M_{1} \times K_{0})) + N_{q}(N).
 \end{eqnarray}


 \subsection{$\mathbf{Computation~~of~~N_{q}(\phi^{-1}(M_{1}\times M_{2}^{s}))}$}
 \vspace{1cm}

 Let
 \begin{equation}\label{A}
  A \,=\, \{ [L \oplus L^{-1}] \in K-K_{0}~|~L~\text{and}~L^{-1}\text{are both defined over}~\mathbb{F}_{q}\}
\end{equation}
and
\begin{eqnarray}\label{B}
  B \,=\, (K-K_0) - A.
\end{eqnarray}

Then it is clear that
\begin{equation}\label{N_q(A)}
 N_{q}(A)\, = \,(1/2)(N_{q}(J_{2})-2^{2g_{2}})
\end{equation}
and
\begin{equation}\label{N_q(B)}
 N_{q}(B) \, = \, N_q(K-K_{0}) - N_{q}(A).
\end{equation}

Let
\begin{equation}\label{beta_1}
 \beta_{1} \, = \, \frac{N_{q}(A)}{(q-1)^{2}} + \frac{N_{q}(B)}{q^{2}-1} + \frac{N_{q}(J_{2})-2^{2g_{2}}N_{q}(\mathbb{P}^{g_{2}-2})}{q-1}
\end{equation}
and
\begin{equation}\label{beta_2}
 \beta_{2} \, = \, \frac{2^{2g_{2}}}{N_{q}(\text{GL}~(2,\mathbb{F}_{q}))} + \frac{2^{2g_{2}}N_{q}(\mathbb{P}^{g_{2}-1})}{q(q-1)}.
\end{equation}

By [\cite{B-S},Proposition 3.7],
 \begin{equation}\label{N_q(M_2(S))}
  N_{q}(M_{2}^{s}) \, = \, N_{q}(M_{L}(2,1)) + \frac{q^{g_{2}-1}N_{q}(J_{2})}{q^{2}-1} - (\beta_{1} + \beta_{2}) (q-1),
 \end{equation}
 where $M_{L}(2,1)$ is the moduli space of stable bundles of rank $2$ and degree $1$ with determinant $L.$

\begin{lemma}
  \begin{eqnarray}\label{N_q(M_1*M_2(S))}
    N_{q}(\phi^{-1}(M_{1}\times M_{2}^{s})) & = & N_{q}(M_{1})N_{q}(\text{PGL}~(2,\mathbb{F}_q))[N_{q}(M_{L}(2,1)) \nonumber \\
    & &+ \frac{q^{g_{2}-1}N_{q}(J_{2})}{q^{2}-1} - (\beta_{1} + \beta_{2}) (q-1)].
 \end{eqnarray}
\begin{proof}
 Let $([E_{1}],[E_{2}])\in M_{1}\times M_{2}^{s}.$ As $E_{2}$ is stable, $\text{gr}~(E_{2})\,=\, E_{2}.$ Since $E_{1}$ is also stable, $\text{Aut}~(E_{1})$ and $\text{Aut}~(E_{2})$ will just be non-zero scalars.\\
 We have $$\phi^{-1}([E_{1}],[E_{2}]) \,=\, \{[E_{1},E_{2},\overrightarrow{T}]|\overrightarrow{T}: E_{1}(p)\rightarrow E_{2}(p)~\text{is a linear isomorphism}\},$$
 where $[E_{1},E_{2},\overrightarrow{T}]$ is an isomorphism class with the triple $(E_{1},E_{2},\overrightarrow{T})$ as its representative. Now if two triples $(E_{1},E_{2},\overrightarrow{T})$ and $(E_{1},E_{2},\overrightarrow{S})$ are in the same isomorphism class, then they satisfy the following commutative diagram

 $$\xymatrix{
E_{1}(p) \ar@{->}[r]^{\overrightarrow{T}} \ar@{->}[d]^{\lambda} & E_{2}(p) \ar@{->}[d]^{\mu}
\\E_{1}(p)\ar@{->}[r]^{\overrightarrow{S}} & E_{2}(p).
}$$
Hence we have $\overrightarrow{S} \,=\, (\frac{\mu}{\lambda})\overrightarrow{T},$
where $\lambda \in \text{Aut}~(E_{1})$ and $\mu \in \text{Aut}~(E_{2}).$\\
Conversely if two triples $(E_{1},E_{2},\overrightarrow{T})$ and $(E_{1},E_{2},\overrightarrow{S})$ are such that $\overrightarrow{S} = \lambda \overrightarrow{T},$ where $\lambda $ is a non-zero scalar, then again by the following commutative diagram

$$\xymatrix{
E_{1}(p) \ar@{->}[r]^{\overrightarrow{T}} \ar@{->}[d]^{1} & E_{2}(p) \ar@{->}[d]^{\lambda}
\\E_{1}(p)\ar@{->}[r]^{\overrightarrow{S}} & E_{2}(p),
}$$
the triples $(E_{1},E_{2},\overrightarrow{T})$ and $(E_{1},E_{2},\overrightarrow{S})$ are in the same isomorphism class. Hence two triples $(E_{1},E_{2},\overrightarrow{T})$ and $(E_{1},E_{2},\overrightarrow{S})$ are in same isomorphism class if and only if $\overrightarrow{S}$ is a scalar multiple of $\overrightarrow{T}$. Now if we fix a basis for $E_{1}(p)$ and $E_{2}(p)$, then $\overrightarrow{T}$ can be identified with an element of $\text{GL}~(2,\mathbb{F}_q).$ But by the above arguments, the triples $(E_{1},E_{2},\overrightarrow{T})$ and $(E_{1},E_{2},\overrightarrow{\lambda T })$ are in same isomorphism class for every non-zero scalar $\lambda.$ This proves that each distinct isomorphism class of triples in $\phi^{-1}([E_{1}],[E_{2}])$ corresponds to a unique element of $\text{PGL}~(2,\mathbb{F}_q).$
So
 \begin{equation}\label{N_q(M_1*M_2(S))_temporary}
 N_{q}(\phi^{-1}(M_{1}\times M_{2}^{s})) \,=\, N_{q}(M_{1}) N_{q}(M_{2}^{s}) N_{q}(\text{PGL}~(2,\mathbb{F}_q)).
 \end{equation}

Now using equation \eqref{N_q(M_2(S))} in the equation \eqref{N_q(M_1*M_2(S))_temporary}, we get the equation \eqref{N_q(M_1*M_2(S))}.
This proves the Lemma.
\end{proof}
\end{lemma}

\subsection{$\mathbf{Computation~~of~~N_{q}(\phi^{-1}(M_{1} \times (K-K_{0})))}$}

Let $([E_{1}], [L \oplus L^{-1}]) \in M_{1} \times (K-K_{0}).$ Our aim is to find out the fiber $\phi^{-1}([E_{1}],[L \oplus L^{-1}]).$ Suppose $E_{2}$ is a vector bundle on $X_{2}$ defined over $\mathbb{F}_{q}$ such that $\text{gr}~(E_{2}) \,\simeq\, L \oplus L^{-1}.$ Every such $E_{2}$ can be given as an extension
\begin{equation}\label{Equation 3.4}
 0\rightarrow L \rightarrow E_{2} \rightarrow L^{-1} \rightarrow 0,
\end{equation}
 where the extension \eqref{Equation 3.4} and the line bundles $L$ and $L^{-1}$ are defined over $\overline{\mathbb{F}}_{q}$  (which is same as $\overline{\mathbb{F}}_{l}$).\\

We know that $(K-K_0) \, = \, A \sqcup B,$ where $A$ and $B$ are as defined in the equations \eqref{A} and \eqref{B} respectively. Therefore
\begin{equation}\label{N_q(M_1*(K-K_0))_temporary}
 N_q(\phi^{-1}(M_1 \times K-K_0)) \,=\, N_q(\phi^{-1}(M_1 \times A)) + N_q(\phi^{-1}(M_1 \times B)).
\end{equation}
We now compute the expressions for $N_q(\phi^{-1}(M_1 \times A))$ and $N_q(\phi^{-1}(M_1 \times B))$ separately.

\begin{lemma}\label{N_q(M_1*A)}
 \begin{equation*}
  N_q(\phi^{-1}(M_1 \times A)) \,= \, N_{q}(M_{1}) [N_{q}(A)N_{q}(\frac{\text{GL}~(2,\mathbb{F}_{q})}{\mathbb{G}_{m} \times \mathbb{G}_{m}}) + 2N_{q}(A)N_{q}(\mathbb{P}^{g_{2}-2})N_{q}(\text{PGL}~(2,\mathbb{F}_{q})).
 \end{equation*}
\begin{proof}
 Suppose $E_{2}$ is a vector bundle on $X_2$ defined over $\mathbb{F}_{q}$ such that $\text{gr}~(E_2)\,\simeq\, L \oplus L^{-1} \in A.$ Then the extension \eqref{Equation 3.4} that corresponds to $E_{2}$ either splits or does not split. We now analyze the fibers in both the cases separately.

\textbf{Case~(i):}
 If the extension \eqref{Equation 3.4} splits, then $E_{2} \,\simeq\, L \oplus L^{-1}.$ In this case $\text{Aut}~(E_{2}) \,\simeq\, \mathbb{G}_{m} \times \mathbb{G}_{m}$ \cite[Lemma 3.3(b)]{B-S}. Now if $(E_{1},E_{2},\overrightarrow{T})$ and $(E_{1},E_{2},\overrightarrow{S})$ belong to the same isomorphism class, then by the following commutative diagram
 $$\xymatrix{
  E_{1}(p) \ar@{->}[r]^{\overrightarrow{T}} \ar@{->}[d]^{\lambda} & E_{2}(p) \ar@{->}[d]^{\psi}
  \\E_{1}(p)\ar@{->}[r]^{\overrightarrow{S}} & E_{2}(p)
}$$
we have $$\overrightarrow{S} \,=\, (\frac{\psi}{\lambda})\overrightarrow{T},$$
where $\lambda \in \text{Aut}~(E_{1})$ and $\psi \in \text{Aut}~(E_{2}).$ Conversely if two triples $(E_{1},E_{2},\overrightarrow{T})$ and $(E_{1},E_{2},\overrightarrow{S})$ are such that $\overrightarrow{S} \,=\, \gamma \overrightarrow{T},$ where $\gamma \in \mathbb{G}_{m} \times \mathbb{G}_{m},$ then one can easily see that they both belong to the same isomorphism class. Hence two triples $(E_{1},E_{2},\overrightarrow{T})$ and $(E_{1},E_{2},\overrightarrow{S})$ are in the same isomorphism class if an only if $\overrightarrow{S}$ is a $\mathbb{G}_{m} \times \mathbb{G}_{m}$ multiple of $\overrightarrow{T},$ i.e., if and only if both $\overrightarrow{S}$ and $\overrightarrow{T}$ represent the same coset of $\frac{\text{GL}~(2,\mathbb{F}_{q})}{\mathbb{G}_{m} \times \mathbb{G}_{m}}.$  So for a fixed $L,$ each distinct isomorphism class of triples
$[E_{1}, L \oplus L^{-1},\overrightarrow{T}]$ corresponds to a unique element in $\frac{\text{GL}~(2,\mathbb{F}_{q})}{\mathbb{G}_{m} \times \mathbb{G}_{m}}.$ Hence \textbf{Case~(i)} contributes $N_{q}(M_{1})N_{q}(A)N_{q}(\frac{\text{GL}~(2,\mathbb{F}_{q})}{\mathbb{G}_{m} \times \mathbb{G}_{m}})$ to $N_{q}(\phi^{-1}(M_{1} \times A)).$\\

\textbf{Case~(ii) :}
If the extension \eqref{Equation 3.4} does not split, then $\text{Aut}~(E_{2}) \,\simeq\, \mathbb{G}_{m}$ \cite[Lemma 3.3(a)]{B-S}. For each fixed isomorphism class of bundles $E_{2}$, repeating the arguments done for the fiber over the stable locus $M_{1} \times M_{2}^{s},$ one can conclude that the set of distinct isomorphism classes of triples of the form $[E_{1},E_{2},\overrightarrow{T}]$ is in bijection with $\text{PGL}~(2,\mathbb{F}_{q}).$ But then each distinct isomorphism class of bundles $E_{2}$ corresponds to an isomorphism class of non-split extensions of type \eqref{Equation 3.4} and each such isomorphism class of non-split extensions corresponds to a unique element of $(\text{Ext}^{1}(L^{-1},L) \backslash \{0\})$ mod $\mathbb{G}_{m}.$ It is well known that $\text{Ext}^{1}(L^{-1},L)$ mod $\mathbb{G}_{m}$ is same as $\mathbb{P}^{g_{2}-2}.$ Since the extensions of $L$ by $L^{-1}$ and the extensions of $L^{-1}$ by $L$ give rise to distinct isomorphism classes of bundles, \textbf{Case~(ii)} contributes $N_{q}(M_{1})2N_{q}(A)N_{q}(\mathbb{P}^{g_{2}-2})N_{q}(\text{PGL}~(2,\mathbb{F}_{q}))$ to $N_{q}(\phi^{-1}(M_{1} \times A)).$

Combining both the cases we get
\begin{equation*}
  N_q(\phi^{-1}(M_1 \times A)) \,= \, N_{q}(M_{1}) [N_{q}(A)N_{q}(\frac{\text{GL}~(2,\mathbb{F}_{q})}{\mathbb{G}_{m} \times \mathbb{G}_{m}}) + 2N_{q}(A)N_{q}(\mathbb{P}^{g_{2}-2})N_{q}(\text{PGL}~(2,\mathbb{F}_{q})).
 \end{equation*}
which proves the Lemma.
\end{proof}
\end{lemma}

\begin{lemma}\label{N_q(M_1*B)}
 \begin{equation*}
  N_q(\phi^{-1}(M_1 \times B)) \, = \, N_q(M_1)N_{q}(B)\frac{N_{q}(\text{GL}~(2,\mathbb{F}_{q}))}{q^{2}-1}.
 \end{equation*}

\begin{proof}
Suppose $E_{2}$ is a vector bundle on $X_2$ defined over $\mathbb{F}_{q}$ such that $\text{gr}~(E_2)\, \simeq\,L \oplus L^{-1} \in B.$ Then from the arguments in \cite[Remark 3.2]{B-S}, the corresponding extension \eqref{Equation 3.4} splits over $\mathbb{F}_{q^{2}}$ (i.e., $L$ and $L^{-1}$ are defined over $\mathbb{F}_{q^{2}}$ and $E_2 \, \simeq \, L \oplus L^{-1}$ over $\mathbb{F}_{q^{2}}$). By the following commutative diagram
$$\xymatrix{
E_{1}(p) \ar@{->}[r]^{\overrightarrow{T}} \ar@{->}[d]^{\lambda} & E_{2}(p) \ar@{->}[d]^{\tau}
\\E_{1}(p)\ar@{->}[r]^{\overrightarrow{S}} & E_{2}(p)
}$$ we have
$$\overrightarrow{S}\, = \, (\frac{\tau}{\lambda})\overrightarrow{T},$$ where $\lambda \in \text{Aut}~(E_{1})$ and $\tau \in \text{Aut}~(E_{2}).$ Since the group $\text{Aut}~(E_2)$ acts on $\mathcal{L}(E_1(p),E_2(p))$ (where $\mathcal{L}(E_1(p),E_2(p))$ denotes the set of all linear isomorphisms from $E_1(p)$ to $E_2(p)$) by composition, we can conclude that two triples $(E_1,E_2,\overrightarrow{T})$ and $(E_1,E_2,\overrightarrow{S})$ are in the same isomorphism class if and only if $\overrightarrow{T}$ and $\overrightarrow{S}$ are in the same orbit under this action. So it is clear that
the set of distinct isomorphism classes of triples will be in bijection with the set of all orbits under this action which can be identified with $\frac{\text{GL}~(2,\mathbb{F}_q)}{\text{Aut}~(E_2)}.$ So the fibers over $M_{1}\times B$ contribute $N_{q}(M_{1})N_{q}(B)N_{q}(\frac{\text{GL}~(2,\mathbb{F}_{q})}{\text{Aut}~(E_2)})$ to $N_{q}(\phi^{-1}(M_{1} \times B)).$ By \cite[Lemma 3.5]{B-S}, we know that  $N_q(\text{Aut}~(E_{2})) \, = \, N_{q^{2}} (\mathbb{G}_{m}) \, = \, q^{2}-1.$
Therefore we have
\begin{equation*}
  N_q(\phi^{-1}(M_1 \times B)) \, = \, N_q(M_1)N_{q}(B)\frac{N_{q}(\text{GL}~(2,\mathbb{F}_{q}))}{q^{2}-1}.
 \end{equation*}
\end{proof}
\end{lemma}
This completes the Lemma.

\begin{lemma}
\begin{eqnarray}\label{Expression_for_N_q(K-K0)}
  N_{q}(\phi^{-1}(M_{1} \times (K-K_{0}))) &=& N_{q}(M_{1})[N_{q}(\tilde{K})(q^{2}-q)+ N_{q}(J_{2})[q + q(q^2-1)N_{q}(\mathbb{P}^{g_{2}-2})] \nonumber \\
  & & - 2^{2g_{2}}[q+ q(q^{2}-1)N_{q}(\mathbb{P}^{g_{2}-2})+ (q^{2}-q)N_{q}(\mathbb{P}^{g_{2}-1})]],
\end{eqnarray}
where $\tilde{K}$ is the canonical desingularization of $K$.
\begin{proof}
 Using the Lemmas \ref{N_q(M_1*A)} and \ref{N_q(M_1*B)} in the equation \eqref{N_q(M_1*(K-K_0))_temporary}, we get
 \begin{eqnarray}\label{N_q(M_1*(K-K_0))_bigexpression}
  N_{q}(\phi^{-1}(M_{1} \times (K-K_{0}))) &=& N_{q}(M_{1}) [N_{q}(A)N_{q}(\frac{\text{GL}~(2,\mathbb{F}_{q})}{\mathbb{G}_{m} \times \mathbb{G}_{m}}) + 2N_{q}(A)N_{q}(\mathbb{P}^{g_{2}-2})N_{q}(\text{PGL}~(2,\mathbb{F}_{q})) \nonumber \\
  & &+ N_{q}(B)\frac{N_{q}(\text{GL}~(2,\mathbb{F}_{q}))}{q^{2}-1}] \nonumber \\
  &=& N_{q}(M_{1})[N_{q}(A)q(q+1) + 2N_{q}(A)N_{q}(\mathbb{P}^{g_{2}-2})q(q^{2}-1) \nonumber \\
  & &+N_{q}(B)(q^{2}-q)].
\end{eqnarray}
 By \cite[Remark 4.4]{B-S}, it is known that
 \begin{equation}\label{N_q(K-K_0)}
  N_{q}(K-K_{0}) \, = \, N_{q}(\tilde{K}) - 2^{2g_{2}} N_{q}(\mathbb{P}^{g_{2}-1}).
 \end{equation}
(One can see \cite{SP} for more details on $K$ and $\tilde{K}$).
Now using \eqref{N_q(A)}, \eqref{N_q(B)} and \eqref{N_q(K-K_0)} in the equation \eqref{N_q(M_1*(K-K_0))_bigexpression}, we get the equation \eqref{Expression_for_N_q(K-K0)}. This completes the Lemma.
\end{proof}
\end{lemma}

\subsection{$\mathbf{Computation~~of~~N_{q}(\phi^{-1}(M_{1} \times K_{0}))}$}

\begin{lemma}
 \begin{equation*}
N_{q}(\phi^{-1}(M_{1} \times K_{0})) \, = \, N_{q}(M_{1})2^{2g_{2}}[1 +(q^{2}-1)N_{q}(\mathbb{P}^{g_{2}-1})].
\end{equation*}
 \begin{proof}
Let $([E_{1}],[L\oplus L]) \in M_{1} \times K_{0}.$ Suppose $E_{2}$ is a vector bundle over $X_{2}$ defined over $\mathbb{F}_{q}$ such that $\text{gr}~(E_{2}) \, \simeq\, L \oplus L.$ Again every such $E_{2}$ can be given as an extension
\begin{equation}\label{Equation 3.8}
 0\rightarrow L \rightarrow E_{2} \rightarrow L \rightarrow 0,
\end{equation}
with $L^{2} \,\simeq\, \mathcal{O}_{X_{2}}.$ Such line bundles are $2^{2g_{2}}$ in number and so, without loss of generality we can assume that $L$ and the extensions of type \eqref{Equation 3.8} are defined over $\mathbb{F}_{q}$.
We again consider two cases, the case when \eqref{Equation 3.8} splits and when it does not split.

\textbf{Case~(i):} If the extension \eqref{Equation 3.8} splits, then $E_{2} \, \simeq \, L \oplus L.$ Also in this case $\text{Aut}~(E_{2}) \, \simeq \, \text{GL}~(2,\mathbb{F}_q)$ \cite[Proposition 3.1(1)]{B-S}.  We claim that for a fixed $L,$ (fixing a basis for $E_{1}(p)$ and $E_{2}(p)$) there is only one isomorphism class of triples namely $[E_{1}, L \oplus L ,\overrightarrow{I_{2}}],$ where $I_{2}$ stands for $2 \times 2$ identity matrix, i.e., $I_{2}$  corresponds to the invertible linear transformation $\overrightarrow{I_{2}}$ which sends the fixed basis of $E_{1}(p)$ to the fixed basis of $E_{2}(p).$

 Suppose $[E_{1}, L \oplus L ,\overrightarrow{T}]$ is an isomorphism class of triples . Fix a basis $\{ v_{1},v_{2}\} $ of $E_{1}(p)$ and a basis $\{ w_{1},w_{2}\}$ of $E_{2}(p).$ Once the bases are fixed, $\overrightarrow{T}$ can be identified with an element of $\text{GL}~(2,\mathbb{F}_q),$ call it $T.$ Consider the following commutative diagram-
$$\xymatrix{
E_{1}(p) \ar@{->}[r]^{T} \ar@{->}[d]^{1} & E_{2}(p) \ar@{->}[d]^{T^{-1}}
\\E_{1}(p)\ar@{->}[r]^{I_{2}} & E_{2}(p).
}$$
Here 1 denotes the identity element of $\text{Aut}~(E_{1}),$  and $T^{-1} \in \text{Aut}~(E_{2})$ is the inverse of the matrix $T.$ From the above diagram, we have $[E_{1},L \oplus L,\overrightarrow{T}] \, = \, [E_{1},L \oplus L,\overrightarrow{I_{2}}].$ So there is only one isomorphism class of triples once $L$ is fixed. Since there are $2^{2g_{2}}$ line bundles of this type, \textbf{Case~(i)} contributes $N_{q}(M_{1})2^{2g_{2}}$ to $N_{q}(\phi^{-1}(M_{1} \times K_{0})).$

\textbf{Case~(ii):} Suppose the extension \eqref{Equation 3.8} does not split. Then $\text{Aut}~(E_{2}) \, \simeq \, \mathbb{G}_{m} \times \mathbb{G}_{a}$ \cite[Proposition 3.1(2)]{B-S}. As earlier, each distinct isomorphism class of bundles $E_{2}$ corresponds to an isomorphism class of non-split extensions of type \eqref{Equation 3.8} and each non-split extension class corresponds to a unique element of $(\text{Ext}^{1}(L,L)\backslash \{0\})$ mod $\mathbb{G}_{m}$ which is nothing but $\mathbb{P}^{g_{2}-1}.$ So, for a fixed $L$ such that $L \oplus L \in K_{0},$ each distinct isomorphism class of triples $[E_{1},E_{2},\overrightarrow{T}]$ (where $E_{2}$ is given by an isomorphism class of non-split extensions of type \eqref{Equation 3.8}), corresponds to a unique element of $\mathbb{P}^{g_{2}-1} \times \frac{\text{GL}~(2,\mathbb{F}_q)}{\mathbb{G}_{m} \times \mathbb{G}_{a}}.$ Finally since $K_{0}$ has $2^{2g_{2}}$ elements, \textbf{Case~(ii)} contributes $2^{2g_{2}}N_{q}(M_{1})N_{q}(\mathbb{P}^{g_{2}-1})N_{q}(\frac{\text{GL}~(2,\mathbb{F}_q)}{\mathbb{G}_{m} \times \mathbb{G}_{a}})$ to $N_{q}(\phi^{-1}(M_{1} \times K_{0})).$ So we have

\begin{eqnarray*}
N_{q}(\phi^{-1}(M_{1} \times K_{0}))\, = \, N_{q}(M_{1})2^{2g_{2}}[1 +N_{q}(\frac{\text{GL}~(2,\mathbb{F}_q)}{\mathbb{G}_{m}\times\mathbb{G}_{a}})N_{q}(\mathbb{P}^{g_{2}-1})].
\end{eqnarray*}

This implies
\begin{equation}\label{Expression_for_N_q(K_0)}
N_{q}(\phi^{-1}(M_{1} \times K_{0})) \, = \, N_{q}(M_{1})2^{2g_{2}}[1 +(q^{2}-1)N_{q}(\mathbb{P}^{g_{2}-1})].
\end{equation}
\end{proof}
\end{lemma}

\subsection{$\mathbf{Computation~~of~~N_{q}(N)}$}
\begin{lemma}
\begin{equation}\label{Expression for N_q(N)}
N_{q}(N)\, = \, N_{q}(M_{1}) N_{q}(M_{2}(-1)) N_{q}(\mathbb{P}^{1})N_{q}(\mathbb{P}^{1}),
\end{equation}
where by $M_2(-1)$ we mean the moduli space of stable bundles of rank two and degree -1 with fixed determinant.
\begin{proof}
 By Lemma \ref{divisor}, $N \,\simeq\, N_1 \times N_2,$ where $N_1$ and $N_2$ are as described in the notation in \S \ref{notation}. Therefore
 \begin{equation}\label{N_q(N)= N_q(P_1) N_q(P_2)}
  N_q(N)\,= \, N_q(N_1)N_q(N_2).
 \end{equation}
 By the choice of weights, it is automatic that

 (i) parabolic stability coincides with the semi-stability of the underlying bundle and

 (ii) underlying bundle is stable implies  it is parabolic stable for any quasi-parabolic structure.

  Then it is known that the canonical map $N_1 \rightarrow M_1$ (which sends a parabolic vector bundle $E_{1_{*}} \in N_1$ to the underlying bundle $E_1 \in M_1$) is a $\mathbb{P}^{1}-$fibration over $M_1,$ locally trivial in Zariski topology. So we have
 \begin{equation}\label{N_q(P_1)}
   N_q(N_1) \,=\, N_q(\mathbb{P}^{1}) N_q(M_1).
 \end{equation}
 Now to find $N_q(N_2),$ let $E_{2_{*}} \in N_2$ be arbitrary. Then we have a canonical surjection $E_2(p) \rightarrow \frac{E_2(p)}{F^{2}E_{2}(p)},$ where $F^{2}E_{2}(p)$ is a one-dimensional subspace of $E_2(p).$ Let $W$ be the kernel of this map. Then clearly $W$ is locally free of rank two and degree -1. By \cite[Proposition 6]{B}, $W$ is stable and the map $\psi:N_2 \rightarrow M_2(-1)$ which sends $E_2$ to $W$ is a $\mathbb{P}^{1}-$fibration, locally trivial in Zariski topology. Therefore
 \begin{equation}\label{N_q(P_2)}
  N_q(N_2) \, = \, N_q(\mathbb{P}^{1})N_q(M_2(-1)).
 \end{equation}
 Using equations \eqref{N_q(P_1)} and \eqref{N_q(P_2)} in the equation \eqref{N_q(N)= N_q(P_1) N_q(P_2)}, we get the equation \eqref{Expression for N_q(N)}.
\end{proof}
\end{lemma}

Using equations \eqref{N_q(M_1*M_2(S))}, \eqref{Expression_for_N_q(K-K0)}, \eqref{Expression_for_N_q(K_0)} and \eqref{Expression for N_q(N)} in the equation \eqref{Equation 3.2}, we have
\begin{eqnarray}\label{Expression_for_N_q(M12)}
N_{q}(M_{12}) &=& N_{q}(M_{1})N_{q}(\text{PGL}~(2,\mathbb{F}_q)[N_{q}(M_{L}(2,1)) + \frac{q^{g_{2}-1}N_{q}(J_{2})}{q^{2}-1} - (\beta_{1} + \beta_{2}) (q-1)] \nonumber \\
& &+ N_{q}(M_{1})[N_{q}(\tilde{K})(q^{2}-q)+ N_{q}(J_{2})[q + q(q^2-1)N_{q}(\mathbb{P}^{g_{2}-2})] \nonumber \\
 & &- 2^{2g_{2}}[q+ q(q^{2}-1)N_{q}(\mathbb{P}^{g_{2}-2})+ (q^{2}-q)N_{q}(\mathbb{P}^{g_{2}-1})]] \nonumber \\
& &+ N_{q}(M_{1})2^{2g_{2}}[1 +(q^{2}-1)N_{q}(\mathbb{P}^{g_{2}-1})] \nonumber \\
& &+ N_{q}(M_{1}) N_{q}(M_{2}(-1)) N_{q}(\mathbb{P}^{1})N_{q}(\mathbb{P}^{1}).
\end{eqnarray}

\section{The Poincar$\acute{\text{e}}$ Polynomial Computation}
Let $R$ be an algebraic number ring and $\mathfrak q$ be a maximal ideal of $R$ contracting to a nonzero maximal ideal $p\mathbb Z$. The field $R/\mathfrak q$ is a finite extension of $\mathbb Z/p\mathbb Z$. We denote it by $\mathbb F_{q}$. Let $X_{R}$ be a projective curve over $R$ with two smooth components $X_{1,R}$ and $X_{2,R}$ intersecting transversally at a nodal point $p_R$. Let $X_{\mathbb F_{q}}$ and $X_{\mathbb C}$ be the curves obtained by base changing $X_R$ to $\mathbb F_{q}$ and $\mathbb C$ respectively.

Fix an odd integer $\chi$ and a polarization $(a_1,a_2)$ such that $a_1\chi$ is not an integer. Let $\xi_R\,=\,(L_{1,R},\mathcal O_{X_{2,R}})$ be a line bundle  on $X_R$, where $L_{1,R}$ is a line bundle of degree $1$ on $X_{1,R}$. Let $\xi_{\mathbb F_q}$ and $\xi_{\mathbb C}$ be the corresponding line bundles over $X_{\mathbb F_q}$ and $X_{\mathbb C}$ respectively. Let $M_{12,R}$ (resp. $M_{12,\mathbb F_q}$ and $M_{12,\mathbb C}$) be the moduli space of stable torsion free sheaves with Euler characteristic $\chi$ and fixed determinant $\xi_{R}$ (resp. $\xi_{\mathbb F_q}$ and $\xi_{\mathbb C}$), where $\chi\,=\,\chi_1 +\chi_2 -2$ and $\chi_1$, $\chi_2$ satisfy inequalities \eqref{4}.  Since $M_{12,R}$ itself is a fine moduli space (cf. \S 2), by the base change property $M_{12,R} \times \text{Spec}(\mathbb F_q)\,\simeq M_{12,\mathbb F_q}$  and $M_{12,R} \times \text{Spec}(\mathbb C) \,\simeq M_{12,\mathbb C}$. Similar description holds for $M_{21,R}$ also. For notational convenience, from now on we drop the subscripts $\mathbb F_q$ and $\mathbb C$ and write $M_{12}$ (resp. $M_{21}$). Subscripts will be understood from the context.

Now we recall few generalities which are used in computing Poincar$\acute{\text{e}}$ polynomial using Weil conjectures. Let $Y_R$ be a smooth projective variety of dimension $n$ defined over an algebraic number ring $R$ and $Y_{\mathbb{C}}$ (resp. Y) be a smooth projective variety obtained from $Y_{R}$ by base changing to $\mathbb{C}$ (resp. $\mathbb F_{q}$).
Following Weil, we define the $zeta$ function of $Y$ as
\begin{eqnarray*}
Z(Y,t) &=& \text{exp}~(\sum_{r=0}^{\infty}N_{q^{r}} \frac{t^{r}}{r}).
\end{eqnarray*}
It is clear that $Z(Y,t)$ is a formal power series with rational coefficients. Now the Weil Conjectures state that  $Z(Y,t)$ can be written in the form
$$Z(Y,t)\, = \, \frac{P_{1}(t)P_{3}(t)\cdots P_{2n-1}(t)}{P_{0}(t)P_{2}(t)\cdots P_{2n}(t)}.$$
where $P_{0}(t) \, = \, (1-t);$ $ P_{2n}(t) \, = \, (1-q^{n}t);$ and for each $1 \leq i \leq (2n-1),$ $P_{i}(t)$ is a polynomial with integer coefficients, which can be written as $$P_{i}(t) \, = \, \prod (1-\beta_{ij}t),$$ where $\beta_{ij}$'s are algebraic integers with $|{\beta_{ij}}| = q^{\frac{i}{2}}.$\\
Furthermore, the Betti numbers of $Y_{\mathbb{C}}$ are given by $B_{i}(Y_{\mathbb{C}}) \, = \, \text{\text{deg}}~(P_{i}(t)).$\\
For $r \geq 1,$ and $Y$ a projective variety over $\mathbb{F}_{q}$ of dimension $n,$ set $\widetilde{N_{q^{r}}}(Y) \,=\, q^{-rn}N_{q^{r}}(Y).$
\par
 We recall the following lemma of Kirwan which is the main tool in the computation of the Poincar$\acute{\text{e}}$ Polynomial of $M_{12}$-

\begin{lemma}\cite[Page-186]{KIR}
Suppose that $Y_{1}, \ldots, Y_{k}$ are smooth projective varieties over~ $\mathbb{F}_{q}$ obtained as reduction mod p of varieties defined in characteristic 0. Suppose that $f$ is a rational function in $(k+1)$ variables with $\mathbb{Z}$ coefficients such that $$f(q^{-r},\widetilde{N_{q^{r}}}(Y_{1}), \ldots, \widetilde{N_{q^{r}}}(Y_{k})) \, = \, 0,$$ for all $r \geq 1.$ Then $$f(t^{2},P_{t}(Y_{1}), \ldots , P_{t}(Y_{k})) \, = \, 0,$$ where $P_{t}(Y_{i})$ is the Poincar$\acute{\text{e}}$ Polynomial of the corresponding variety over $\mathbb{C}.$
\end{lemma}

 Multiplying both sides of the equation \eqref{Expression_for_N_q(M12)} by $q^{-\text{\text{\text{\text{dim}}}}~(M_{12})}$ and using the fact that $N_{q}(\text{PGL}~(2,\mathbb{F}_q)) \, = \, q(q^{2}-1),$ and $N_{q}(\mathbb{P}^{1})\, = \, (q+1),$ and simplifying, we have
\begin{eqnarray}\label{Expression_for_N_q(M12)_tilde}
\widetilde{N_{q}}(M_{12})&=& \widetilde{N_{q}}(M_{1})(1-q^{-2})[\widetilde{N_{q}}(M_{L}(2,1)) + \widetilde{N_{q}}(J_{2})\{\frac{q^{-g_{2}}}{1-q^{-2}}\}
 - q^{-3g_{2}+3}(\beta_{1} + \beta_{2})(q-1)] \nonumber \\
& &  + \widetilde{N_{q}}(M_{1}) 2^{2g_{2}} [q^{-3g_{2}} + q^{-2g_{2}+1}(1-q^{-2})\widetilde{N_{q}}(\mathbb{P}^{g_{2}-1})] \nonumber \\
& &+ \widetilde{N_{q}}(M_{1})[\widetilde{N_{q}}(\tilde{K}) q^{-2g_{2}+2} (1-q^{-1}) + \widetilde{N_{q}}(J_{2})[q^{-2g_{2}+1} + q^{-g_{2}+1}(1-q^{-2})\widetilde{N_{q}}(\mathbb{P}^{g_{2}-2})] \nonumber \\
& &-2^{2g_{2}} [q^{-3g_{2}+1} + q^{-2g_{2}+1}(1-q^{-2})\widetilde{N_{q}}(\mathbb{P}^{g_{2}-2}) + q^{-2g_{2}+1}(1-q^{-1})\widetilde{N_{q}}(\mathbb{P}^{g_{2}-1})] \nonumber \\
& &+ \widetilde{N_{q}}(M_{1})\widetilde{N_{q}}(M_{2}(-1))q^{-1}(1+2q^{-1}+q^{-2})].
\end{eqnarray}
It is to be noted that all the varieties that occur in the equation \eqref{Expression_for_N_q(M12)_tilde} are smooth projective varieties. Further, replacing $q$ by $q^r$ in the equation \eqref{Expression_for_N_q(M12)_tilde}, one can get the expression for $\widetilde{N_{q^{r}}}(M_{12})$, for $r \, \geq \,1$.

We can now apply Kirwan's Lemma and obtain the Poincar$\acute{\text{e}}$ Polynomial of $M_{12}.$
\begin{theorem}\label{P_t(M_12)}
The Poincar$\acute{\text{e}}$ Polynomial of $M_{12}$ is given as follows:
\begin{eqnarray*}
P_{t}(M_{12})&=& P_{t}(M_{1})(1-t^{4})[P_{t}(M_{L}(2,1))+ P_{t}(J_{2})\{\frac{t^{2g_{2}}}{1-t^{4}}\} - (\widetilde{\beta_{1}(t)} + \widetilde{\beta_{2}(t)})]\\
& &+ P_{t}(M_{1})2^{2g_{2}}[t^{6g_{2}} + t^{4g_{2}-2}(1-t^{4})P_{t}(\mathbb{P}^{g_{2}-1})]\\
& & + P_{t}(M_{1})[P_{t}(\tilde{K}) t^{4g_{2}-4}(1-t^{2}) + P_{t}(J_{2})[t^{4g_{2}-2} + t^{2g_{2}-2}(1-t^{4})P_{t}(\mathbb{P}^{g_{2}-2})]\\
& &-2^{2g_{2}}[t^{6g_{2}-2} + t^{4g_{2}-2}(1-t^{4})P_{t}(\mathbb{P}^{g_{2}-2})+ t^{4g_{2}-2}(1-t^{2})P_{t}(\mathbb{P}^{g_{2}-1})]]\\
& &+ P_{t}(M_{1})P_{t}(M_{2}(-1))t^{2}(1+2t^{2}+t^{4}),
\end{eqnarray*}

where
\begin{eqnarray*}
P_{t}(M_{1}) &=& \frac{(1+t^{3})^{2g_{1}}-t^{2g_{1}}(1+t)^{2g_{1}}}{(1-t^{2})(1-t^{4})}, \\
P_{t}(M_{L}(2,1)) &=& \frac{(1+t^{3})^{2g_{2}}-t^{2g_{2}}(1+t)^{2g_{2}}}{(1-t^{2})(1-t^{4})} \,=\, P_{t}(M_{2}(-1)), \\
P_{t}(J_{2}) &=& (1+t)^{2g_{2}},\\
\widetilde{\beta_{1}(t)} &=& P_{t}(\tilde{K})[\frac{t^{4g_{2}-4}}{(1+t^{2})}] + P_{t}(J_{2})[\frac{t^{4g_{2}-2}}{(1-t^{4})} + P_{t}(\mathbb{P}^{g_{2}-2})t^{2g_{2}-2}] \\
& &-P_{t}(\mathbb{P}^{g_{2}-1})[2^{2g_{2}} \frac{t^{4g_{2}-2}}{(1+t^{2})}] - 2^{2g_{2}}[\frac{t^{6g_{2}-2}}{(1-t^{4})} + P_{t}(\mathbb{P}^{g_{2}-2})t^{4g_{2}-2}],\\
\widetilde{\beta_{2}(t)} &=& 2^{2g_{2}}[\frac{t^{6g_{2}}}{(1-t^{4})} + t^{4g_{2}-2} P_{t}(\mathbb{P}^{g_{2}-1})],\\
P_{t}(\mathbb{P}^{n}) &=& (1+t^{2}+\cdots+t^{2n}),\\
P_{t}(\tilde{K}) &=& \sum _{i=0}^{2g_{2}}b_{i}(\tilde{K})t^{i},
\end{eqnarray*}
where\\
\begin{eqnarray*}
b_{i}(\tilde{K}) &=& 0 ,~\text{if i is odd},\\
&=& 1,~\text{if i= 0 or}~ 2g_{2},~ \text{and} \\
&=& {2g_{2} \choose i} + 2^{2g_{2}}  , ~ \text{if i is even and}~ 0 < i < 2g_{2}.
\end{eqnarray*}
\end{theorem}

(see [\cite{SP}, Theorem 2]).\\

We give below the Betti numbers of the variety $M_{12}$ for some low genus curves.\\

\begin{center}
\begin{longtable} [c]{|c|c|c|c|c|c|c|}
\caption{\label{long}}\\
\hline

\multicolumn {7} {|c|} {Betti Nos. of $M_{12}$ for low genus} \\
\hline
\bf Betti & \bf $g_{1}=3$ & \bf $g_{1}=3$ & \bf $g_{1}=3$ & \bf $g_{1}=4$ & \bf $g_{1}=4$ & \bf $g_{1}=4$\\
\bf No. &   \bf $g_{2}=3$ & \bf $g_{2}=4$ & \bf $g_{2}=5$ & \bf $g_{2}=3$ & \bf $g_{2}=4$ & \bf $g_{2}=5$ \\
\hline
\endfirsthead

\hline
\multicolumn {7} {|c|} {Continuation of Table \ref{long}} \\
\hline
\bf Betti & \bf $g_{1}=3$ & \bf $g_{1}=3$ & \bf $g_{1}=3$ & \bf $g_{1}=4$ & \bf $g_{1}=4$ & \bf $g_{1}=4$\\
\bf No. &   \bf $g_{2}=3$ & \bf $g_{2}=4$ & \bf $g_{2}=5$ & \bf $g_{2}=3$ & \bf $g_{2}=4$ & \bf $g_{2}=5$ \\
\hline
\endhead

\hline
 \endfoot

 \hline
 \multicolumn{7}{| c |}{End of Table}\\
 \hline\hline
 \endlastfoot

$B_{0}$ & 1 & 1 & 1 & 1 & 1 & 1 \\

\hline
$B_{1}$ & 0 & 0 & 0 & 0 & 0 & 0 \\

\hline
$B_{2}$ & 3 &  3 & 3 & 3 & 3 & 3 \\

\hline
$B_{3}$ & 12 & 14 & 16 & 14 & 16 & 18 \\

\hline
$B_{4}$ & 8 & 8 & 8 & 8 & 8 & 8 \\

\hline
$B_{5}$ & 36 & 42 & 48 & 42 & 48 & 54 \\

\hline
$B_{6}$ & 81 & 106 & 135 & 107 & 136 & 169 \\

\hline
$B_{7}$ & 90 & 106 & 122 & 112 & 128 & 144 \\

\hline
$B_{8}$ & 207 & 284 & 371 & 299 & 388 & 487 \\

\hline
$B_{9}$ & 356 & 542 & 768 & 568 & 808 & 1096 \\

\hline
$B_{10}$ & 435 & 657 & 879 & 708 & 974 & 1240 \\

\hline
$B_{11}$ & 698 & 1224 & 1850 & 1320 & 2040 & 2880 \\

\hline
$B_{12}$ & 992 & 1953 & 3203 & 2089 & 3524 & 5336 \\

\hline
$B_{13}$ & 1120 & 2412 & 3988 & 2592 & 4520 & 6756 \\

\hline
$B_{14}$ & 1345 & 3520 & 6606 & 3781 & 7448 & 12302 \\

\hline
$B_{15}$ & 1520 & 4784 & 9976 & 5110 & 11152 & 19872 \\

\hline
$B_{16}$ & 1345 & 5386 & 12271 & 5731 & 13821 & 25555 \\

\hline
$B_{17}$ & 1120 & 6236 & 16900 & 6626 & 18920 & 39008 \\

\hline
$B_{18}$ & 992 & 6884 & 22191 & 7314 & 24621 & 55967 \\

\hline
$B_{19}$ & 698 & 6236 & 24804 & 6626 & 27496 & 68762 \\

\hline
$B_{20}$ & 435 & 5386 & 28060 & 5731 & 30996 & 90304 \\

\hline
$B_{21}$ & 356 & 4784 & 30512 &  & 33584 & 114202 \\

\hline
$B_{22}$ & 207 &  & 28060 &  & 30996 & 126671 \\

\hline
$B_{23}$ & 90 &  & 24804 &  & 27496 & 140424 \\

\hline
$B_{24}$ & 81 &  &  &  &  & 150346 \\

\hline
$B_{25}$ & 36 &  &  &  &  & 140424 \\

\hline
$B_{26}$ & 8 &  &  &  &  &  \\

\hline
$B_{27}$ & 12 &  &  &  &  &  \\

\hline
$B_{28}$ & 3 &  &  &  &  &  \\

\hline
$B_{29}$ & 0 &  &  &  &  &  \\

\hline
$B_{30}$ & 1 &  &  &  &  &  \\

\hline
\end{longtable}
\end{center}

\subsection{The Poincar$\acute{\text{e}}$ Polynomial of $M_{21}$}
We now briefly sketch the changes that need to be incorporated in order to find the Poincar$\acute{\text{e}}$ polynomial of $M_{21}$ (see \S \ref{notation} for the definition of $M_{21}$).

Let $L_1^{\prime}$ be the invertible sheaf on $X_1$ such that $L_1^{\prime^{\otimes 2}} \,=\, L_1^{*} \otimes \mathcal{O}_{X_1}(-p).$ This is possible because $L_1^{*} \otimes \mathcal{O}_{X_1}(-p)$ is of even degree. Let $L_2^{\prime} \,=\, \mathcal{O}_{X_{2}}(p).$ Now tensoring $M_{21}$ by the invertible sheaf $L^{\prime} \,=\, (L_1^{\prime}, L_2^{\prime},\overrightarrow{\lambda^{\prime}}),$ one gets an isomorphism between $M_{21}$ and $M_{21}^{\prime},$ where
\begin{eqnarray*}
 M_{21}^{\prime} &=& \{[F_{1}^{\prime},F_{2}^{\prime},\overleftarrow{S^{\prime}}] :~~[F_{1}^{\prime},F_{2}^{\prime},\overleftarrow{S^{\prime}}]~
 \text{is an isomorphism class of stable triples} ,\\
& & \text{rk}~(F_{i}^{\prime}) \,=\, 2,~\Lambda^{2}(F_{1}^{\prime}) \,\simeq\, \mathcal{O}_{X_1},~\Lambda^{2}(F_{2}^{\prime}) \,\simeq\, \mathcal{O}_{X_{2}}(p)\}.
\end{eqnarray*}
So finding the Poincar$\acute{\text{e}}$ Polynomial of $M_{21}$ is equivalent to finding the Poincar$\acute{\text{e}}$ Polynomial of $M_{21}^{\prime}.$
Now analogous to the natural map $\Phi$ which we had in \eqref{fi}, we have a surjective map
$$\Theta:M_{21}^{\prime} \rightarrow M_{1}^{\prime} \times M_{2}^{\prime}$$
which sends equivalence class of the triple $[F_{1}^{\prime},F_{2}^{\prime},\overleftarrow{S^{\prime}}]$ to $([F_{1}^{\prime}],[F_{2}^{\prime}]),$ where
\begin{eqnarray*}
  M_{1}^{\prime} & = & \{ [F_{1}^{\prime}] :~[F_{1}^{\prime}]~\text{is the gr-equivalence class of semi-stable vector bundles on}~X_{1}, \\  & & \text{rk}~(F_{1}^{\prime}) \,=\, 2,~\Lambda^{2}(F_{1}^{\prime}) \,\simeq\,  \mathcal{O}_{X_1}\},
\end{eqnarray*}
and
\begin{eqnarray*}
  M_{2}^{\prime}& = & \{ [F_{2}^{\prime}] :~[F_{2}^{\prime}]~\text{is the isomorphism class of stable bundles
  on}~X_{2},~\text{rk}~(F_{2}^{\prime}) \,=\, 2, \\
 & & \Lambda^{2}(F_{2}^{\prime}) \,\simeq\, \mathcal{O}_{X_{2}}(p) \}.
\end{eqnarray*}
Now we apply exactly the same stratification for $M_{1}^{\prime}$ as we did for $M_{2}$ earlier and repeat the same arguments as before to obtain the Poincar$\acute{\text{e}}$ polynomial of $M_{21}^{\prime}$ and hence for $M_{21}.$

\begin{remark}\label{euler_characteristic}
(i)It is easy to see that the Poincar$\acute{\text{e}}$ polynomial of $M_{12}$ for the case $g_{1}\, = \,i,~g_{2} \,= \,j$ is same as that of $M_{21}$ for the case
$g_{1} \,= \,j,~g_{2} \,=\, i,$ where $i,j \in \mathbb{Z}, i \geq 3, j\geq 3.$

(ii) The topological Euler characteristic $\chi$ of the spaces $M_{12}$ and $M_{21}$ can be obtained by evaluating the respective Poincar$\acute{\text{e}}$ polynomials at $t\,=\,-1.$ By Theorem \ref{P_t(M_12)} and the above remark, it follows that $\chi(M_{12}) \,=\, 0 \,=\, \chi(M_{21}).$
\end{remark}

\section{The Intersection Poincar$\acute{\text{e}}$ Polynomial Computation}

 By \cite[Lemma 2.4]{S-B}, it is known that the fixed determinant moduli space $M$ is a union of two irreducible, smooth varieties $M_{12}$ and $M_{21}$ intersecting transversally along the smooth divisor $N.$ Since both the irreducible components of $M$ are smooth and are of dimension $3g-3,$ by \cite[Theorem 19.2]{Wh}, $M$ has a Whitney stratification which makes it into a topological pseudomanifold (\cite[IV \S 2]{Bo}).

 Let $\widetilde{M},\widetilde{M_{12}}~\text{and}~\widetilde{M_{21}}$ be the algebraic normalizations of $M, M_{12}~ \text{and}~M_{21}$ respectively. Since $M_{12}$ and $M_{21}$ intersect transversally along $N,$ we have
 \begin{eqnarray}\label{tilde(M)}
  \widetilde{M} &=& \widetilde{M_{12}} \sqcup \widetilde{M_{21}} \nonumber \\
            & \simeq & M_{12} \sqcup M_{21}.
 \end{eqnarray}

 Now since $\widetilde{M},~ M_{12}~\text{and}~M_{21}$ are normal, their intersection homology groups (with top perversity $\overline{t}$) will coincide with the singular homology groups by the Theorem \ref{(G-M)_Theorem_4.3}. So Theorem \ref{(G-M)_Theorem_4.3} combined with the isomorphism \eqref{tilde(M)} gives
\begin{equation}\label{IH(tilde(M))_top_perversity}
 IH_{*}^{\overline{t}}(\widetilde{M}) \,\simeq\, H_{*}(\widetilde{M}) \,\simeq \, H_{*}(M_{12}) \oplus H_{*}(M_{21}).
\end{equation}

Now using the Remark \ref{normalization_related_remark} and the Theorem \ref{(G_M)_Theorem_4.2} in the equation \eqref{IH(tilde(M))_top_perversity}, we get
\begin{equation}\label{IH(M)_top_perversity}
 IH_{*}^{\overline{t}}(M) \,\simeq\, H_{*}(M_{12}) \oplus H_{*}(M_{21})
\end{equation}

Let $\text{IP}^{\overline{t}}~(M)$ denote the intersection Poincar$\acute{\text{e}}$ polynomial of $M$ with respect to the top perversity $\overline{t}$ (see definition \ref{definition_of_IPP}). Then by the equation \eqref{IH(M)_top_perversity}, we have

\begin{theorem}\label{IPP(M)}
\begin{equation*}
 \text{IP}^{\overline{t}}(M) \,=\, P_t(M_{12}) + P_t(M_{21}),
\end{equation*}
where $P_t(M_{12})~ \text{and}~P_t(M_{21})$ are the Poincar$\acute{\text{e}}$ polynomials of $M_{12}$ and $M_{21}$ that have been computed in the previous section.
\end{theorem}

\begin{remark}\label{intersection_euler_characteristic}
  We notice that, since $\chi(M_{12})$ and $\chi(M_{21})$ are zero, the intersection Euler characteristic of $M$ also turns out to be zero by Theorem \ref{IPP(M)}.
\end{remark}


\begin{thebibliography}{99}

\bibitem{A-B} Atiyah, M.F.; Bott, R.; \emph{The Yang-Mills equations over Riemann surfaces},~~~
Philos. Trans. Roy. Soc. London Ser. A 308 (1983), no. 1505, 523615.

\bibitem{B} Balaji, V.; \emph{Cohomology of ceratin moduli spaces of vector bundles},~~~ Proc. Indian Acad. Sci (Math. Sci.), Vol. 98, No. 1, June 1988, pp. 1-24.

\bibitem{B-S} Balaji, V.; Seshadri, C.S.; \emph{Poincar$\acute{\text{e}}$ polynomials of some moduli varieties}, ~~~
Algebraic geometry and analytic geometry (Tokyo, 1990), 125, ICM-90 Satell. Conf. Proc., Springer, Tokyo, 1991.

\bibitem{B-D-S} Barik, Pabitra; Dey, Arijit; Suhas, B.N.; \emph{On the Rationality of Nagaraj-Seshadri Moduli Space},~~~ Accepted for publication in Bulletin Des Sciences Mathematiques.


\bibitem{U-R} Desale, Usha V.; Ramanan, S.; \emph{Poincar$\acute{\text{e}}$ polynomials of the variety of stable bundles}, Math. Ann. 216
(1975), no. 3, 233244.

\bibitem{H-N} Harder, G.; Narasimhan, M.S.; \emph{On the cohomology groups of moduli spaces of vector bundles on
curves}, ~~~ Math. Ann. 212 (1974/75), 215248.

\bibitem{KIR} Kirwan, F.C.; \emph{Cohomology of quotients in symplectic and algebraic geometry},~~~ Mathematical Notes 31, Princeton.Univ.Press(1984).

\bibitem{N-S} Nagaraj, D.S.; Seshadri, C.S.; \emph{Degenerations of moduli spaces of vector bundles on curves -I}, ~~~ Proc. Indian Acad. Sci (Math. Sci.),
Vol-107, No:2, 1997.

\bibitem{Serre}  Serre, J.-P.; On the fundamental group of a unirational variety.  J. London Math. Soc. 34 1959 481--484.
\bibitem{SP} Spanier , E.H.; \emph{The Homology of Kummer Manifolds}, ~~~ Proc.Amer.Soc. 7 (1956), 155-160.

\bibitem{S-B} Basu , Suratno; \emph{On a Relative Mumford-Newstead Theorem}, ~~~ arXiv:1501.07347v2 [math.AG], To appear in Bulletin des Sciences Mathematiques.

\bibitem{G-M} Goresky , Mark; MacPherson, Robert; \emph{Intersection Homology Theory},~~~ Topology, Vol. 19. pp. 135-162.

\bibitem{G_M-II} Goresky , Mark; MacPherson, Robert; \emph{Intersection Homology II},~~~  Invent. Math. 71, 77-129.

\bibitem{K-W} Kirwan , Frances; Woolf , Jonathan; \emph{An Introduction to Intersection Homology Theory},~~~ Second Edition, Chapman and Hall/CRC, Boca Raton, FL, 2006.


\bibitem{Wh} Whitney, H; \emph{Tangents to an analytic variety},~~~Ann.Math.,81:496-549,1965

\bibitem{Bo} Borel, A $et~al$; \emph{Intersection cohomology},~~~volume 50 of Progress in Mathematics. Birkh$\ddot{a}$user, 1984. Notes on the seminar held at the University of Bern, 1983.

\end{thebibliography}
\end{document}